\documentclass[12pt]{article}
\usepackage{amsmath,amssymb,amsthm,multicol,multirow,tabularx,color}
\usepackage[longnamesfirst]{natbib}
\usepackage{tikz}
\usetikzlibrary{arrows}
\usepackage{caption}
\usepackage[margin=2cm]{geometry}
\usepackage{aliascnt}
\usepackage[linktocpage]{hyperref}
\hypersetup{
bookmarksnumbered={true}
}
\usepackage[all]{hypcap}
\usepackage{abbrevs}
\usepackage{bhmvs2}

\newcommand{\nr}[1]{\nameref{#1} (\ref{#1})}

\newcommand{\ffield}{\operatorname{Frac}}

\author{Andrew Critch, UC Berkeley\footnote{This research was supported by the DARPA Deep Learning program (FA8650-10-C-7020)}}
\title{Binary Hidden Markov Models and Varieties}

\date{July 9, 2012}

\makeatletter
\let\c@table\c@theorem
\let\c@figure\c@theorem
\makeatother

\def\makethm#1#2{%
  \newaliascnt{#1}{theorem}%
  \newtheorem{#1}[#1]{#2}%
  \aliascntresetthe{#1}%
  \expandafter\providecommand\csname #1autorefname\endcsname{#2}%
}

\theoremstyle{plain}
\makethm{thm}{Theorem}
\makethm{prop}{Proposition}
\makethm{conj}{Conjecture}
\makethm{cor}{Corollary}
\makethm{lem}{Lemma}
\makethm{claim}{Claim}

\theoremstyle{definition}
\makethm{defn}{Definition}
\makethm{rem}{Remark}
\makethm{rems}{Remarks}
\makethm{que}{Question}
\makethm{conv}{Convention}

\theoremstyle{remark}
\makethm{nb}{N.B.}

\makethm{fact}{Fact}

\usepackage{bbm}

\newcommand{\one}{\mathbbm{1}}

\begin{document}
\maketitle
\abstract{\small
The technological applications of hidden Markov models have been extremely diverse and successful, including natural language processing, gesture recognition, gene sequencing, and Kalman filtering of physical measurements.  HMMs are highly non-linear statistical models, and just as linear models are amenable to linear algebraic techniques, non-linear models are amenable to commutative algebra and algebraic geometry.

This paper closely examines HMMs in which all the hidden random variables are binary.  Its main contributions are (1) a {\em birational parametrization} for every such HMM, with an explicit inverse for recovering the hidden parameters in terms of observables, (2) a semialgebraic model membership test for every such HMM, and (3) {\em minimal defining equations} for the 4-node fully binary model, comprising 21 quadrics and 29 cubics, which were computed using Gr\"obner bases in the cumulant coordinates of \citeauthor*{SZ11}.  The new model parameters in (1) are rationally identifiable in the sense of \citeauthor*{GSS10}, and each model's Zariski closure is therefore a {\em rational} projective variety of dimension $5$.  Gr\"obner basis computations for the model and its graph are found to be considerably faster using these parameters.
In the case of two hidden states, item (2) supersedes a previous algorithm of \citeauthor*{Sch11} which is only generically defined, and the defining equations (3) yield new invariants for HMMs of all lengths $\geq 4$.  Such invariants have been used successfully in model selection problems in phylogenetics, and one can hope for similar applications in the case of HMMs.
}



\section{Introduction}

The present work is motivated primarily by the problems of {\em model selection} and {\em parameter identifiability}, viewed from the perspective of algebraic geometry.  By beginning with the simplest hidden Markov models (HMMs) --- those where all hidden nodes are binary --- the hope is that eventually a very precise geometric understanding of HMMs can be attained that provides insight into these central problems.  Indeed, most questions about this case are answered by reducing to the case where the visible nodes are also binary.  The history of this and related problems has two main branches of historical lineage: that of hidden Markov models, and that of algebraic statistics.  

Hidden Markov models were developed as statistical models in a series of papers by Leonard E. Baum and others beginning with \citet{BP66}, after the description by \citet{Str60} of the \lq\lq{}forward-backward\rq\rq{} algorithm that would be used for HMM parameter estimation.  
HMMs have been used 
extensively in natural language
processing and speech recognition since the development of DRAGON by \citet{Bak75}.  
As well, since \citet{KMH94} used HMM for gene finding in the DNA of in E. coli bacteria, they have had many applications in genomics and biological sequence alignment; 
see also \citep{Yoo09}.  
Now, HMM parameter estimation is built into the measurement of so many kinds of time-series data that it would be gratuitous to enumerate them.  However, the methods of algebraic statistics are not so old, and the algebraic geometry of these models is far from fully explored.  They are hence an important early example for the theory to investigate.

Algebraic statistics is the application of commutative algebra and algebraic geometry to the study of statistical models, especially those models involving non-linear relations between parameters and observables.  It was first described at length in the monograph {\em Algebraic Statistics} by \citet*{PGR01}\footnote{\citeauthor{PGR01} attribute their interest in the subject to a seminar paper of \citet{DS98} circulated as a manuscript in 1993, which employed Gr\"obner bases to construct Markov random walks.}.
Subsequent introductions to the subject include {\em Algebraic Statistics for Computation Biology} by \citet*{PS05}, and {\em Lectures in Algebraic Statistics} by \citet*{DMS09}.  Also notable is {\em Algebraic Geometry and Statistical Learning Theory} by \citet*{Wat09}, for its focus on the problem of model selection from data.  

To the problem of {\em model selection}, the algebraic analogue is {\em implicitization}, i.e., finding polynomial defining equations for the Zariski closures of binary hidden Markov models.  Such polynomials are called {\em invariants} of the model: if a polynomial $f$ is equal to a constant $c$ at every point on the model (i.e. does not vary with the model parameters), then we encode this equation by calling $f-c$ {\em an invariant}.  Model selection and implicitization are more than simply analogous; polynomial invariants have been used successfully in model selection by \citet{CFS06} and \citet{Eri08} for phylogenetic trees. 

Invariants have been difficult to classify for hidden Markov models, perhaps due to the high codimension of the models.  \citet*{BM05} found many invariants using linear algebra, but did not exhibit any generating sets of invariants, and in fact their search was actually for invariants of a model that was slightly modified from the HMM proper.
\citet{Sch11} found a large family of invariants arising as minors of certain non-abelian Hankel matrices, and was able to verify that such invariants generate the ideal of the 3-node binary HMM, the simplest non-degenerate HMM.  However, this seemed not to be the case for models with $n\geq 4$ nodes: Sch\"onhuth reported on a computation of J. Hauenstein which verified numerically that the $4$-node model was not cut out by the Hankel minors.  

In \ar{sec-def}, we will make use of moment and cumulant coordinates as exposited in \citep{SZ11}, as well as a new coordinate system on the parameter space, to find explicit defining equations for the $4$-node binary HMM.  The shortest quadric and cubic equations are fairly simple; to give the reader a visual sense, they look like this:
\begin{itemize}
\item[] $g_{{2},1}=m_{23}m_{13}-m_{2}m_{134}-m_{13}m_{12}+m_{1}m_{124}$
\item[] $g_{{3},1}=m_{12}^3-2m_{1}m_{12}m_{123}+m_{\emptyset}m_{123}^2+m_{1}^2m_{1234}-m_{\emptyset}m_{12}m_{1234}$
\end{itemize}
Here each $m$ is a moment of the observed probability distribution.  These equations are not generated by Sch\"onhuth's Hankel minors, and so provide a finer test for membership to any binary HMM of length $n\geq 4$ after marginalizing to any 4 consecutive nodes.

To the problem of {\em parameter identifiability}, the algebraic analogue is the generic or global injectivity or finiteness of a map of varieties that parametrizes the model, or in the case of identifying a single parameter, constancy of the parameter on the fibers of the parameterization.  \citet{GSS10} provide an excellent discussion of this topic in the context of identifying causal effects; see also \citep{MED09} for a striking application to identification for ODE models in the biosciences.

In \ar{sec-bir}, for the purpose of parameter identification in binary hidden Markov models, we express the parametrization of a binary HMM as the composition of a dominant and generically finite monomial map $\fq$ and a birationally invertible map $\psi$.  An explicit inverse to $\psi$ is given, which allows for the easy recovery of hidden parameters in terms of observables.  The components of the monomial map are {\em identifiable combinations} in the sense of \citet{MED09}.  The formulae for recovering the hidden parameters are fairly simple when exhibited in a particular order, corresponding to a particular triangular set of generators in a union of lexicographic Gr\"obner bases for the model ideal.  To show their simplicity, the most complicated recovery formula looks like this:
$$u = \frac{m_1m_3-m_2^2+m_{23}-m_{12}}{2(m_3-m_2)}$$
As a corollary, in \ar{sec-intbir} we find that the fibers of $\phi_n$ are generically zero-dimensional, consisting of two points which are equivalent under a ``hidden label swapping'' operation.

\ar{sec-trace} describes how the parametrization of every fully binary HMM, or ``BHMM", can be factored through a particular $9$-dimensional variety called a {\em trace variety}, which is the invariant theory quotient of the space of triples of $2\times 2$ matrices under a simultaneous conjugation action by $SL_2$.  As a quotient, the trace variety is not defined inside any particular ambient space.  However, its coordinate ring, a {\em trace algebra}, was found by \citet{Sib68} to be generated by $10$ elements, which means we can embed the trace variety in $\CC^{10}$.  We prove the main results of \ar{sec-bir} in the coordinates of this embedding.  As a byproduct of this approach, in section \ar{sec-birproof} we find that the Zariski closures of all BHMMs with $n\geq 3$ are birational to each other.

Finally, \ar{sec-app} explores some applications of our results, including model membership testing, classification of identifiable parameters, a new grading on HMMs that can be used to find low-degree invariants, the geometry of equilibrium BHMMs, and HMMs with more than two visible states.

I would like to thank my advisor, Bernd Sturmfels, and postdoctoral mentor, Shaowei Lin, for many helpful conversations and editorial suggestions on this paper.

\section{Definitions}\label{sec-back}

{\bf Important note:} \ In this paper, we will work mostly with BHMMs --- HMMs in which both the hidden and visible nodes are all binary --- because, as will be explained in \ar{sec-hmms}, all our results will generalize to allow $\geq 2$ visible states by reducing to this case.

Throughout, we will be referring to binary hidden Markov {\em processes, distributions, maps, models, varieties,} and {\em ideals}.  Each of these terms is used with a distinct meaning, and effort is made to keep their usages consistent and separate.

\subsection{Binary Hidden Markov processes and distributions}\label{sec-prodis}

A binary hidden Markov process is a statistical process which generates random binary sequences.  It is based on the simpler notion of a binary (and not hidden)  Markov {\em chain} process.  

\begin{defn}\label{def-BMCP} A {\bf Binary Hidden Markov {\em process}} will comprise 5 data: $\pi$, $T$, $E$, and $(H_t,V_t)$.  The pair $(H_t,V_t)$ denotes a jointly random sequence $(H_1,V_1,H_2,V_2,\ldots$) of binary variables, also respectively called {\em hidden nodes} and {\em visible nodes}, with range $\labs$.  Often a bound $n$ on the (discrete) time index $t$ is also given.    The joint distribution of the nodes is specified by the following:

\begin{itemize}
\item A row vector ${\pi = (\pi_0, \pi_1)}$, called the {\em initial distribution}, which specifies a probability distribution on the first hidden node $H_1$ by $\pr(H_1=i)=\pi_i$;
\item A matrix ${T= \bmat{
T_{00} & T_{01}  \\ 
T_{10} & T_{11}}}$, called the {\em transition matrix}, which specifies conditional ``transition'' probabilities by the formula $\pr(H_t=j\, | \, H_{t-1}=i) = T_{ij}$, read as the probability of ``transitioning from hidden state $i$ to hidden state $j$''.\footnote{\citep{Sch11} uses $T$ for different matrices, which I will later denote by $P$.}
\item A matrix ${E= \bmat{
E_{00} & E_{01}  \\ 
E_{10} & E_{11}}}$, called the {\em emission matrix}, which specifies conditional ``emission'' probabilities by the formula $\pr(V_t=j\, | \, H_{t}=i) = E_{ij}$, read as the probability that ``hidden state $i$ emits the visible state $j$''.
\end{itemize}
\end{defn}

To be precise, the parameter vector $\theta=(\pi,T,E)$ determines a probability distribution on the set of sequences of pairs $\left((H_1,V_1)\ldots(H_n,V_n)\right) \in (\labs ^2)^n$, or if no bound $n$ is specified, a compatible sequence of such distributions as $n$ grows.   In applications, only the joint distribution on the visible nodes $(V_1,\ldots,V_n)\in\labs ^n$ is observed, and is called the {\em observed distribution}.  This distribution is given by marginalizing (summing) over the possible hidden states of a BHM process:
\begin{align}
\pr(V=v\, |\, \theta=(\pi, T, E)) &= \sum_{h\in \labs ^n} \pr(h,v | \, \pi,T,E) = \sum_{h\in \labs ^n} \pr(h\, |\, \pi,T)\pr(v\, |\, h,E)\notag\\ 
&= \sum_{h\in \labs ^n} \pi_{h_1}E_{h_1,v_1} \prod_{i=2}^n T_{h_{i-1}h_i}E_{h_i,v_i}\label{eqn-bhmmap}
\end{align}
\begin{defn}\label{bhmdist} A {\bf Binary Hidden Markov {\em distribution}} is a probability distribution on sequences $v\in \labs^n$ of jointly random binary variables $(V_1,\ldots,V_n)$ which arises as the observed distribution of {\em some} BHM {\em process} according to \eqref{eqn-bhmmap}.
\end{defn} 
As we will see in \ar{sec-linpar}, different processes $\proc$ can give rise to the same observed distribution on the $V_t$, for example by permuting the labels of the hidden variables, or by other relations among the parameters.  

Those already familiar with Markov models in some form may note that:
\begin{itemize}{
\item The data $(\pi,T,H_t)$ alone specify what is ordinarily called a binary Markov {\em chain} process on the nodes $H_t$.  In the applications we have in mind, these nodes are unobserved variables.
\item The matrices $T$ and $E$ are assumed to be {\em stationary}, meaning that they are not allowed to vary with the ``time index'' $t$ of $(H_t,V_t)$.
\item The distribution $\pi$ {\em is not} assumed to be {\em at equilibrium}, i.e.\,we {\em do not} assume that $\pi T = \pi$.  This allows for more diverse applications.
}\end{itemize}
\begin{nb}The term ``stationary'' is sometimes also used for a process that is at equilibrium; we will reserve the term ``stationary'' for the constancy of matrices $T$,$E$ over time.\end{nb}

\subsection{Binary Hidden Markov maps, models, varieties, and ideals}\label{sec-mmvi}

Statistical processes come in families defined by allowing their parameters to vary, and in short, the set of probability distributions that can arise from the processes in a given family is called a {\em statistical model}.  The Zariski closure of such a model in an appropriate complex space is an algebraic variety, and the geometry of this variety carries information about the purely algebraic properties of the model.

In a binary hidden Markov process, $\pi$, $T$, and $E$ must be {\em stochastic matrices}, i.e. each of their rows must consist of non-negative reals which sum to $1$, since these rows are probability distributions.  We denote by $\tsim$ the set of such triples $(\pi,T,E)$, which is isometric to the 5-dimensional cube $(\Delta_1)^5$.  We call $\tsim$ the space of {\em stochastic parameters}.  It is helpful to also consider the larger space of triples $(\pi,T,E)$ where the matrices can have arbitrary complex entries with row sums of $1$.  We write $\taff$ for this larger space, which is equal to complex Zariski closure of $\tsim$, and call is the space of {\em complex parameters}. 

{\em We will not simply replace $\tsim$ by $\taff$ for convenience}, as has sometimes been done in algebraic phylogenetics.  For the ring of polynomial functions on these spaces, we write
\begin{align}
\tring &:= \CC[\pi_j, T_{ij}, E_{ij}] \bigg/ \Big(1=\sum_j \pi_j = \sum_j T_{ij} = \sum_j E_{ij}\te{ for }i=0,1\Big)\notag
\end{align}
so as to make the identification $\tsim \sbs \taff = \spec \tring \notag$.  Here $\spec$ denotes the spectrum of a ring; see \citep*{CLO07} for this and other fundamentals of algebraic geometry.

Now we a fix a length  $|v|=n$ for our binary sequences $v$, and write 
\begin{align*}
\pring{n} &:= \CC[p_v \, | \, v\in\labs ^n] & \paff{n} &:= \spec(\pring{n})\\
\spring{n} &:= \pring{n} \big/ (1-\sum_{|v|=n} p_v) & \spaff{n} &:= \spec(\spring{n}) \\[-2ex]
&& \ppro{n} &:=\proj(\pring{n})
\end{align*}
We will often have occasion to consider the natural inclusions,
\begin{align*}
\iota_n: \spaff{n} &\hra \paff{n} & \ol{\iota}_n : \spaff{n} &\hra \ppro{n}
\end{align*}

\begin{conv} Complex spaces such as $\CC^{2^n}$ will usually be decorated with a subscript to indicate the intended coordinates to be used on that space, like the $p$ in $\paff{n}$ above.  Likewise, a ring will usually be denoted by $R$ with some subscripts to indicate its generators.
\end{conv}

\begin{defn}\label{def-bhmmap}For $n\geq 3$,
\begin{itemize}
\item The {\bf Binary Hidden Markov {\em map}} or {\em modeling map} on $n$ nodes is the map $\bhmf{n}$, or simply $\phi_n$, given by given by \eqref{eqn-bhmmap}, i.e.
$$\phi_n: \taff \to \spaff{n},$$
$$\phi_n\us(p_v) := \sum_{h\in \labs ^n} \pi_{h_1}E_{h_1,v_1} \prod_{i=2}^n T_{h_{i-1}h_i}E_{h_i,v_i}$$

The word ``model'' is also frequently used for the map $\phi_n$.  This is a very reasonable usage of the term, but I reserve ``model\rq\rq{} for the image of the allowed parameter values:

\item $\bhmm{n}$, the {\bf Binary Hidden Markov {\em model}} on $n$ nodes, is the image $$\ol\iota_n\phi_n\left(\tsim\right)\sbs \ppro{n},$$ of the stochastic parameter space $\tsim$, \ie, the set of observed distributions which can arise from {\em some} BHM process, considered as a subset of $\ppro{n}$ via $\ol\iota_n$.  Being the continuous image of the classically compact cube $\tsim\simeq \Delta_1^5$, $\bhmm{n}$ is also classically compact and hence classically closed.

\item $\bhmv{n}$, the {\bf Binary Hidden Markov {\em variety}} on $n$ nodes, is the Zariski closure of $\bhmm{n}$, or equivalently the classical closure of $\phi_n(\taff)$, in $\ppro{n}$.  

\item $\bhmi{n}$, the {\bf Binary Hidden Markov {\em ideal}} on $n$ nodes, is the set of homogeneous polynomials which vanish on $\bhmm{n}$, \ie,  the homogeneous defining ideal of $\bhmv{n}$.  Elements of $\bhmi{n}$ are called {\em invariants} of the model.
\end{itemize}
\end{defn}

In summary, probability distributions arise from processes according to modeling maps, models are families of distributions arising from processes of a certain type, and the Zariski closure of each model is a variety whose geometry reflects the algebraic properties of the model.  The ideal of the model is the same as the ideal of the variety: the definition of Zariski closure is the largest set which has the same ideal of vanishing polynomials as the model.  In a rigorous sense (namely, the anti-equivalence of the categories of affine schemes and rings), the variety encodes information about the ``purely algebraic'' properties of the model, i.e. properties that can be stated by the vanishing of polynomials.  

The number of polynomials that vanish on any given set is infinite, but by the Hilbert Basis theorem, one can always find finitely many polynomials whose vanishing implies the vanishing of all the others.  This is called a {\em generating set} for the ideal.  To compute a generating set for $\bhmi{n}$, we will need the following proposition:

\begin{prop}\label{res-hom} The ideal $\bhmi{n}$ is the homogenization of $\ker (\phi_n\us\circ \iota_n\us)$ with respect to $p_\Sigma := \sum_{|v|=n} p_v$

\end{prop}
\begin{proof}
The affine ideal $\ker (\phi_n\us\circ \iota_n\us)$ cuts out the Zariski closure $X$ of $\iota_n\circ\phi_n(\taff)$ in $\paff{n}$, and this closure lies in the hyperplane $\{p_\Sigma=1\} = \spaff{n}$.
Let $X'$ be the projective closure of $X$ in $\ppro{n}$, so that $I(X')$ is the homogenization of $\ker (\phi_n\us\circ \iota_n\us)$ with respect to $p_\Sigma$.

The cube $\tsim$ is Zariski dense in $\taff$, so $\iota_n\circ\phi_n(\tsim)$ is Zariski dense in $\iota_n\circ\phi_n(\taff)$, which is Zariski dense in $X$, which is Zariski dense in $X'$.  Therefore $X'=\bhmv{n}$, and $I(X') = \bhmi{n}$, as required.
\end{proof}

\subsection{HMMs with more visible states via BHMM(n)}\label{sec-hmms}
All the results of this paper apply to HMMs with more than two visible states, using the following trick.  Consider $\hmm{2,k,n}$, an HMM with $2$ hidden states, $k$ visible states $\alpha_1\ldots\alpha_k$, and $n$ (consecutive) visible nodes.  Such a hidden Markov process can be specified by a $2\times k$ matrix $E$ of emission probabilities, along with a $1\times 2$ matrix $\pi$ and a $2\times 2$ matrix $T$ describing the two-state hidden Markov chain as in \eqref{eqn-pite}.  For each $\ell\in\{1\ldots,k\}$, we have a way to interpret this process as a BHM process by letting $\alpha_j=1$ and $\alpha_i=0$ for $i\neq j$.  The resulting binary emission matrix is 
$${E'(\ell)= \bmat{
1-E_{0\ell} & E_{0\ell}  \\ 
1-E_{1\ell} & E_{1\ell}}},$$
so as $\ell$ varies, we obtain all the entries $E_{ij}$ as entries of an $E'(\ell)$.  We shall remark throughout when results can be generalized to $\hmm{2,k,n}$ using this trick.

\section{Defining equations of $\bhmv{3}$ and $\bhmv{4}$} \label{sec-def}

\begin{thm}\label{res-bhmi4}
The homogeneous ideal $\bhmi{4}$ of the binary hidden Markov variety $\bhmv{4}$ is minimally generated by 21 homogeneous quadrics and 29 homogeneous cubics.
\end{thm}

Since \citet{Sch11} found numerically that his Hankel minors did not cut out $\bhmm{4}$ even set-theoretically, these equations are genuinely new invariants of the model.  Moreover, they are not only applicable to $\bhmm{4}$, because a BHM process of length  $n>4$ can be marginalized to any $4$ consecutive hidden-visible node pairs to obtain a BHM process of length $4$.  Thus, we have $n-3$ linear maps from $\bhmm{n}$ to $\bhmm{4}$, each of which allows us to write $21$ quadrics and $29$ cubics which vanish on $\bhmm{n}$.  Finally, using \ar{sec-hmms}, we can even obtain invariants of $\hmm{2,k,n}$ via the $k$ different reductions to $\bhmm{n}$.

Our fastest derivation of \ar{res-bhmi4} in Macaulay2 \citep*{M2} uses the birational parametrization of \ar{sec-bir}, but in only a single step, so we defer the lengthier discussion of the parametrization until then.  Modulo this dependency, the proof is described in \ar{sec-der}, using moment coordinates (\ar{sec-mom}) and cumulant coordinates (\ar{sec-cum}).

In probability coordinates, the generators found for $\bhmi{4}$ had the following sizes:

\begin{itemize}

\item 
Quadrics $g_{2,1},\ldots,g_{2,21}$: respectively $8, 8, 12, 14, 16, 21, 24, 24, 26, 26, 28, 32, 32, 41,$\\
$42, 43, 43, 44, 45, 72, 72$ probability terms.
\item Cubics $g_{3,1},\ldots,g_{3,29}$: respectively $32, 43, 44, 44, 44, 52, 52, 56, 56, 61, 69, 71, 74, 76, 78,$\\
$ 81, 99, 104, 109, 119, 128, 132, 148, 157, 176, 207, 224, 236, 429$ probability terms.
\end{itemize}
As a motivation for introducing moment coordinates, we note here that these generators have considerably fewer terms when written in terms of moments:
\begin{itemize}
\item Quadrics $g_{2,1},\ldots,g_{2,21}$: respectively $4, 4, 4, 4, 6, 6, 6, 6, 6, 6, 6, 6, 8, 8, 8, 8, 8, 10, 10, 10, 17$ moment terms.
\item Cubics $g_{3,1},\ldots,g_{3,29}$: respectively $5, 6, 6, 6, 6, 6, 6, 6, 6, 6, 6, 8, 8, 8, 8, 10, 10, 10, 10, 10, 12,$\\$ 12, 13, 14, 16, 18, 21, 27, 35$ moment terms.
\end{itemize}
To give a sense of how these polynomials look in moment coordinates, the shortest quadric and cubic are
\begin{itemize}
\item $g_{{2},1}=m_{23}m_{13}-m_{2}m_{134}-m_{13}m_{12}+m_{1}m_{124}$, and
\item $g_{{3},1}=m_{12}^3-2m_{1}m_{12}m_{123}+m_{\emptyset}m_{123}^2+m_{1}^2m_{1234}-m_{\emptyset}m_{12}m_{1234}$.
\end{itemize}
Let us compare this ideal with $\bhmi{3}$, the homogeneous defining ideal of $\bhmv{3}$.  \citet{Sch11} found that $\bhmi{3}$ is precisely the ideal of  $3\times 3$ minors of the following matrix: 
\begin{equation} \label{eqn-pmat} A_{3,3}=\bmat{
p_{000}+p_{001} & p_{000} & p_{100} \\ 
p_{010}+p_{011} & p_{001} & p_{101} \\ 
p_{100}+p_{101} & p_{010} & p_{110} \\ 
p_{110}+p_{111} & p_{011} & p_{111}}
\end{equation}

Sch\"onhuth defines an analogous matrix $A_{n,3}$ for $\bhmv{n}$, but then remarks that J. Hauenstein has found, using numerical rank deficiency testing \citep{BHPS10} with the algebraic geometry package Bertini \citep{Ber}, that $\minors_3(A_{n,3})$ does not cut out $\bhmv{n}$ when $n=4$.  In general, Sch\"onhuth shows that $\bhmi{n} =(\minors_3(A_{n,3}):\minors_2(B_{n,2}))$ for a particular $2\times 3$ matrix $B_{n,2}$, but computing generators for this colon ideal is a costly operation, and so no generating set for $\bhmi{n}$ was not found for any $n\geq 4$ by this method.  Instead, here we will make use of {\em moment coordinates} and {\em cumulant coordinates} as exposited in \citep{SZ11}.

\subsection{Moment coordinates}\label{sec-mom}

Moments are particular linear expressions in probabilities.  They can be derived from a moment generating function as in \citep{SZ11}, but in our case, moments can be expressed simply by the following rule: we order $\labs ^n$ by strict dominance, \ie $v\geq w$ iff $v_i \geq w_i$ for all $i$, and  then
\begin{equation} \label{eqn-mom} m_v := \sum_{w\geq v} p_w \in \pring{n} \end{equation}
Since all our variables are binary, with the usual algebraic statistical convention that a ``$+$'' subscript denotes an index to be summed over, we can view the conversion from moments to probabilities as ``replacing zeros by $+$ signs".  For example, $m_{10010} = p_{1++1+}. $ The ring elements $m_v\in \pring{n}$ provide alternative linear coordinates on $\ppro{n}$ in which it turns out that some previously intractable BHM computations are simplified and become feasible.

For a more compact notation, a binary string $v$ of length $n$ is the indicator function of a unique subset $I$ of $[n]=\{1,\ldots,n\}$, so we also write $m_I$ to represent $m_v$.  For example, $m_{0000} = m_\emptyset$, $m_{1000} = m_{1}$, and $m_{0101} = m_{24}$.  
From \eqref{eqn-mom} we can see that $m_I$ actually represents a marginal probability: $m_I = \pr(V_i = 1\textrm{ for all } i \in I)$.  Thus, in the context of BHMMs , no confusion results if we write $m_I$ without specifying the value of $n$.  To be precise, if $I\sbs [n]$ and $I'$ denotes $I$ considered as a subset of $[n']$ for some $n'>n$, then
\begin{equation}\label{eqn-subs}
\phi_n^\#(m_I) = \phi_{n'}^\#(m_{I'})
\end{equation}
This can be seen in many ways, for example using the \nameref{res-baum} (\ar{res-baum}) as explained in \autoref{sec-trunc}.

Just as for probabilities, for moments we define rings and spaces
\begin{align}
\notag \mring{n} &:= \CC[m_I \, | \, I\sbs[n]] & \notag \maff{n} &:= \spec(\mring{n})\\
\label{eqn-rbar}\smring{n} &:= \mring{n} \big/ \langle 1-m_\emptyset \rangle & \smaff{n} &:= \spec(\smring{n})\\
&& \notag \mpro{n} &:=\proj(\mring{n}),
\end{align}
To avoid having notation for too many ring isomorphisms, we adopt:
\begin{conv}\label{conv-mom}Using \eqref{eqn-mom}, we will usually treat $m_I$ as a literal element of $\pring{n}$, thus creating literal identifications 
\begin{equation}\label{eqn-idents}
\mring{n} = \pring{n}, \quad \smring{n} = \spring{n}, \quad \maff{n} = \paff{n}, \quad \mpro{n} = \ppro{n}, \quad \te{and} \quad \smaff{n} = \spaff{n}.
\end{equation}
\end{conv}
Note that, for example, we obtain natural ring inclusions 
\begin{equation*}
\mring{n} \sbs \mring{n'}
\end{equation*}
whenever $n<n'$, which respect the BHM maps $\phi_n$ because of \eqref{eqn-subs}.

As a first application of moment coordinates, we have 
\begin{prop}\label{res-bhmv3} The homogeneous ideal $\bhmi{3}$ is generated in moment coordinates by the $3\times 3$ minors of
the matrix
$$ 
A'_{3,3}=\bmat{
m_{000} & m_{000} & m_{100} \\ 
m_{010} & m_{001} & m_{101} \\ 
m_{100} & m_{010} & m_{110} \\ 
m_{110} & m_{011} & m_{111}}
=\bmat{
m_\es & m_\es & m_{1} \\ 
m_{2} & m_{3} & m_{13} \\ 
m_{1} & m_{2} & m_{12} \\ 
m_{12} & m_{23} & m_{123}}
$$ 
In particular, the projective variety $\bhmv{3}$ is cut out by these minors.
\end{prop}
\begin{proof}
Observe that Sch\"onhuth's matrix $A_{3,3}$ in \eqref{eqn-pmat} is equivalent under elementary row/column operations to $A'_{3,3}$, so $\minors_3 A'_{3,3}=\minors_3 A_{3,3}=\bhmi{3}$.\end{proof}

\begin{prop}\label{res-mhom} The ideal $\bhmi{n}$ is the homogenization of $\ker (\phi_n\us)$ with respect to $m_\es$.
\end{prop}
\newcommand{\kerphiiota}{\ker(\phi_4^\# \circ \iota_4^\#)}
\newcommand{\kerphi}{\ker(\phi_4^\#)}

\begin{proof}
From \autoref{res-hom} we know that $\bhmi{n}$ is the homogenization of $\ker (\phi_n\us\circ \iota_n\us)$ with respect to $m_\es=\sum_{|v|=n} p_v$.  From \eqref{eqn-rbar}, we can identify $\smring{4}$ with the polynomial subring of $\mring{4}$ obtained by omitting $m_\es$, so that 
$\kerphiiota =  \kerphi+ \langle 1-  m_{\es} \rangle$. 
Since the additional generator $1- m_\es $ homogenizes to $0$, $\kerphi$ has the same homogenization as $\kerphiiota$, hence the result.
\end{proof}

\subsection{Cumulant coordinates}\label{sec-cum}
Cumulants are non-linear expressions in moments or probabilities which seem to allow even faster computations with binary hidden Markov models.  Let 
\newcommand{\kring}[1]{R_{k,{#1}}}
\newcommand{\skring}[1]{\ol{R}_{k,{#1}}}
\newcommand{\skaff}[1]{\CC^{2^{#1}-1}_k}
\begin{align*}
\kring{n} &:=\CC[k_I \, | \, I\sbs [n]]\\
\skring{n} &:= \kring{n} \big/ \langle k_\es \rangle\\
\skaff{n} &:= \spec(\skring{n})
\end{align*}
where, as with moments, we may freely alternate between writing $k_v$ and writing $k_I$, where $I$ is the set of positions where $1$ occurs in $v$.  For building generating functions, let $x_1,\ldots,x_n$ be indeterminates, and write $x^v=x^I$ for $x_1^{v_1}\cdots x_n^{v_n} = \prod_{i\in I} x_i$.  Let $J$ be the ideal generated by all the squares $x_i^2$.  Following \citep{SZ11}, we define the {\em moment } and {\em cumulant generating functions}, respectively, as
$$f_m(x) := \sum_{I\sbs [n]}m_I x^I \in \smring{n}[x]/J\qquad \qquad
f_k(x) := \sum_{I\sbs [n]}k_I x^I \in \skring{n}[x]/J
$$
We now define changes of coordinates
$$
\kappa_n: \smaff{n} \to \skaff{n}\qquad \qquad
\kappa_n\-: \skaff{n} \to\smaff{n}
$$
by the formulae
\begin{align}
\label{eqn-cum} \kappa_n^\# (f_k) &= \log(f_m) = \frac{(f_m-1)}{1} + \cdots + (-1)^{n+1}\frac{(f_m-1)^n}{n}\\
\notag \kappa_n^{-\#}(f_m) &= \exp(f_k) = 1+\frac{(f_k)}{1} + \cdots + \frac{(f_k)^n}{n!} 
\end{align}
That is, we let $\kappa_n^\#(k_I)$ be the coefficient of $x^I$ in the Taylor expansion of $\log{f_m}$ about $1$, and let $\kappa_n^{-\#}(m_I)$ be the coefficient of $x^I$ in the Taylor expansion of $\exp{f_k}$ about $0$. Note that in the relevant coordinate rings $\smring{n}$ and $\skring{n}$, $m_{\es}=1$ and $k_{\es}=0$.  This is why we only need to compute the first $n$ terms of each Talyor expansion: the higher terms all vanish modulo the ideal $J$.  
\newcommand{\kerphik}{{\bf I}_{k,4}}
\newcommand{\kerphim}{{\bf I}_{m,4}}
\begin{prop}The expressions $\kappa_n^\#(k_I)$ and $\kappa_n^{-\#}(m_I)$, i.e. writing of cumulants in terms of moments and conversely, do not depend on $n$.
\end{prop}
\begin{proof}
In \citep{SZ11}, these formulae are re-expressed using M\"obius functions, which do not depend on the generating function description above, and in particular do not depend on $n$.
\end{proof}
\subsection{Deriving $\bhmi{4}$ in Macaulay2}\label{sec-der}
This section describes the proof of \ar{res-bhmi4} using Macaulay2.  These computations were carried out on  a Toshiba Satellite P500 laptop running Ubuntu 10.04, with an Intel Core i7 Q740 .73 GHz CPU and 8gb of RAM.  In light of \ar{res-mhom}, we will aim to compute $\kerphiiota$, which can be understood geometrically as the (non-homogeneous) ideal of the standard affine patch of $\bhmv{4}$ where $m_\es=\sum_{|v|=4}p_v=1$.  To reduce the number of variables, as in \ar{res-mhom} we continue to make the identification 
\begin{equation*}
\smring{4} = \CC[m_I | \es \neq I \sbs [4]] \sbs \mring{4}
\end{equation*}

We begin by providing Macaulay2 with the map $\phi_4^\#: \smring{4} \to \tring$
in moment coordinates (\ar{sec-mom}), because probability coordinates result in longer, higher degree expressions.  This can be done by composing the expression of $\phi_n\us(p_v)$ in \ar{def-bhmmap} with the expression of $m_v=m_I$ in \eqref{eqn-mom}, or alternatively using the Baum formula for moments (\ar{res-baum}), which involves many fewer arithmetic operations.

Macaulay2 runs out of memory (8gb) trying to compute $ \kerphi$, and as expected, this memory runs out even sooner in probability coordinates, so we use cumulant coordinates instead (\ar{sec-cum}).  We input 
\begin{equation*}
\kappa_4^\#: \skring{4} \to \smring{4}
\end{equation*}
using coefficient extraction from \eqref{eqn-cum}, and compute the composition 
$\phi_4\us \circ \kappa_4\us$.  Then, it is possible to compute
\begin{equation*}
 \kerphik:=\ker(\phi_4^\# \circ \kappa_4^\#)
\end{equation*}
which takes {\bf around 1.5 hours}. 
Alternatively, we can compute $ \kerphik$ using the birational parameterization $\psi_4$ of \ar{sec-bir} in place of $\phi_4$, which takes {\bf less than 1 second} and yields $100$ generators for $\kerphik$.  

Subsequent computations run out of memory with this set of $100$ generators, so we must take some steps to simplify it.  Macaulay2's {\bf trim} command reduces the number of generators of $ \kerphik$ to $46$ in {\bf under 1 second}.  
We then order these $46$ generators lexicographically, first by degree and then by number of terms, and eliminate redundant generators in reverse order, which takes {\bf 19 seconds}.  
The result is an inclusion-minimal, non-homogeneous generating set for $ \kerphik$ with 35 generators: 24 quadrics and 11 cubics.

Now we compute ${\bf I}_{m,4}:=\kappa\us(I_{k,4})
=\kappa\us(\ker(\phi_4\us\circ\kappa_4\us))
=\ker(\phi_4\us),
$
i.e., we push forward the 35 generators for $\kerphik$ under the non-linear ring isomorphism $\kappa_4^\#$ to obtain 35 generators for $\kerphim=\kerphi$: 2 quadrics, 7 cubics, 16 quartics, 5 quintics, and 5 sextics.  In {\bf under 1 second}, Macaulay2's {\bf trim} command computes a new set of 39 generators for $\kerphim$ with lower degrees: 21 quadrics,  14 cubics, and 4 quartics, which turns out to {\bf save around 1 hour} of computing time in what follows.  These generators have many terms each, and eliminating redundant generators as in the previous paragraph turns out to be too slow to be worth it here, taking more than 2 hours, so we omit this step.

Finally, we apply \ar{res-mhom} to compute $\bhmi{4}$ as the homogenization of $\kerphim$ with respect to $m_\es$.  In Macaulay2, this is achieved by homogenizing the 39 generators for $\kerphim$ with respect to $m_\es$ and then saturating the ideal they generate with respect to $m_\es$.  This saturation operation takes about {\bf 29 minutes}, and yields a minimal generating set of 50 polynomials: 21 quadrics and 29 cubics. 
Since probabilities are linear in moments, their degrees are the same in probability coordinates.   Moreover, since these are homogeneous generators for a homogeneous ideal, they are minimal in a very strong sense:

\begin{cor}
Any inclusion-minimal homogeneous generating set for $\bhmi{4}$ in probability or moment coordinates must contain exactly 21 quadrics and 29 cubics.
\end{cor}

We still do not know a generating set for $\bhmi{5}$.  Macaulay2 runs out of memory (8gb) attempting to compute ${{\bf I}_{k,5}}$, even using the birational parametrization of \ar{sec-bir}.  The author has also attempted this computation using the {\em tree cumulants} of \citet*{SZ10} in place of cumulants, but again Macaulay2 runs out of memory trying to compute the first kernel.  Presumably the subsequent saturation step would be even more computationally difficult.  

\section{Birational parametrization of BHMMs }\label{sec-bir}
\begin{thm}[Birational Parameter Theorem]\label{res-bir}
There is a generically two-to-one, dominant morphism $\taff \to \CC^5$ such that, for each $n\geq 3$, the binary hidden Markov map $\phi_n$ factors uniquely as follows, and each $\psi_n: \CC^5 \to \bhmv{n}$ has a birational inverse map $\rho_n$:
\begin{center}\vspace{-1ex}
\begin{tikzpicture}[node distance=2cm, auto, semithick, >=stealth']
  \node (C) {$\CC^5$};
  \node (P) [right of=C, node distance=2.2cm] {$\spaff{n}$};
  \draw[->] (C) to node {$\psi_n$}	(P);
  \node (Ct) [left of=C] {$\taff$};
  \draw[->, bend right] (Ct) to node [swap] {$\phi_n$}	(P);
  \draw[->] (Ct) to node {}		(C); 
  \node (blank) [right of=P, node distance=3cm] {};
  \node (C) [below of=blank, node distance=0.4cm] {$\CC^5$};
  \node (P) [right of=C, node distance=2.2cm] {$\bhmv{n}$};
  \draw[->, bend left] (C) to node {$\psi_n$}	(P);
  \draw[->, bend left, dashed] (P) to node {$\rho_n$}	(C);
\end{tikzpicture}
\end{center}\vspace{-1ex}
In particular, $\bhmv{n}$ is always a rational projective variety of dimension $5$, \ie, birationally equivalent to $\PP^5$.
\end{thm}
Using the reduction of \ar{sec-hmms}, the same is true if we allow $k > 2$ visible states in the model and replace $5$ by $3+k$.  This theorem will be proven in \ar{sec-birproof} using trace algebras and the Baum formula for moments.  In the course of this section and \ar{sec-trace} we will exhibit formulae for $\psi_n$ and their inverses $\rho_n$.  The inverse map $\rho_3$ has a number of practical uses, to be explored in \ar{sec-app}.

Our first step toward \ar{res-bir} is to re-parametrize $\taff$.

\subsection{A linear reparametrization of $\taff$}\label{sec-linpar}

Since the hidden variables $H_t$ are never observed, there is no change in 
the final expression of $p_v$ in \ar{def-bhmmap} if we swap the  labels $\labs$ of all the $H_t$ 
simultaneously.  This swapping is equivalent to an action of the elementary 
permutation matrix $\sigma=\left(\smat{0 & 1  \\ 1 & 0}\right)$:
\begin{align}
\notag \sw: \taff &\to \taff\\
\label{eqn-sw} \theta=(\pi,T,E) &\mapsto (\pi \sigma, \, \sigma\- T \sigma, \, \sigma\- E)
\end{align}
(In our case $\sigma\- = \sigma$, but the form above generalizes to permutations of larger hidden alphabets.)  Hence we have that 
$\pr(v \, | \, \pi, T, E) = \pr(v \, | \ \sw(\pi , T, E))$, i.e. $\phi_n = \phi_n \circ \sw$.

We will make essential use of a linear parametrization of $\taff$ in which $\sw$ has a simple form.  Our new parameters will be  
 $\eta_0:=(a_0,b,c_0,u,v_0),$ with subscript $0$'s to be explained shortly.  Although we have already used the letter $v$ at times to represent visible binary strings, we hope that the context will be clear enough to avoid confusion between these usages.  We let
\begin{equation}\label{eqn-pite}
\begin{split}
\pi \; = \; \frac{1}{2}\bmat{1 - a_0, & 1 + a_0}\hspace{20ex}\\
T \; = \; \frac{1}{2} \bmat{
1+b-c_0, & 1-b+c_0 \\ 
1-b-c_0, & 1+b+c_0}
\qquad E \; = \; \bmat{
1-u+v_0, & u-v_0  \\ 
1-u-v_0, & u+v_0
} 
\end{split}
\end{equation}
(The rightmost column of $E$ is made intentionally homogeneous in the new parameters.)  We can linearly solve for $\eta_0$ in terms of $\theta$ by $a_0=\pi_1-\pi_0$ etc., so in fact $(a_0,b,c_0,u,v_0)$ generate the parameter ring $\CC[\theta]$.  In these coordinates, $\sw$ acts by
\begin{align*}
a_0 &\mt -a_0, & b \mt b, && c_0 &\mt -c_0, & u \mt u, && v_0 &\mt-v_0
\end{align*}
In other words, swapping the signs of the subscripted variables $a_0,c_0,v_0$ has the same effect as acting on the matrices $\pi,T,E$ by $\sigma$ as in \eqref{eqn-sw}, \ie, relabeling the hidden alphabet.

\subsection{Introducing the birational parameters}\label{}
Since $\phi_n \circ \sw =  \phi_n$, by classical invariant theory the ring map $\phi_n^\#: \spring{n} \to \tring$ must land in the subring of invariants $\tring^{\sw} = \CC[b,u,a_0^2,c_0^2,v_0^2,a_0c_0,a_0v_0,c_0v_0].$
However, $\phi_n^\#$ in fact factors through a smaller subring, conveniently generated by 5 elements:
\begin{lem}[Parameter Subring Lemma]\label{res-sub}
For all $n\geq 3$, the ring map $\phi_n^\#$ lands in the subring 
$$ \nring :=\CC[a,b,c,u,v] $$
of $\tring$, where $a=a_0v_0$, $c=c_0v_0$, $v=v_0^2$.
\end{lem}
The proof of this key lemma will be given in \autoref{sec-sub} after introducing trace algebras.  To interpret its geometric consequences,  write $\fq\us$ for the subring inclusion
$$\fq\us : \nring \hra \tring \phantom{\mt}$$
$$a\mt a_0v_0, \quad b \mt b, \quad c \mt c_0v_0, \quad u \mt u, \quad v \mt v_0^2,$$
write $\psi_n\us:\spring{n} \to \nring$ for the factorization of $\phi_n\us$ through $\fq\us$, and write 
$\naff:= \spec \CC[\eta]$, so $\naff\simeq \CC^5$.  The result:
\begin{cor}\label{res-dom}The following diagram of dominant maps commutes
\begin{center}
\begin{tikzpicture}[node distance=2cm, auto, semithick, >=stealth']
  \node (Cn) {$\naff$};
  \node (Vn) [right of=Cn, node distance=2.5cm] {$\bhmv{n}$};
  \draw[->] (Cn) to node {$\psi_n$}	(Vn);
  \node (Ct) [left of=Cn] {$\taff$};
  \draw[->, bend right] (Ct) to node [swap] {$\phi_n$}	(Vn);
  \draw[->] (Ct) to node {${\mathfrak q}$}		(Cn); 
\end{tikzpicture}
\end{center}
and $\fq$ is generically two-to-one.
\end{cor}
This corollary in particular implies the first part of the \nr{res-bir}, by taking $\fq : \taff \to \naff\simeq \CC^5$ as the generically $2:1$ map.
\begin{rem}The map $\fq$ is only dominant, and not surjective; for example, it misses the point $(1,0,0,0,0)$.  
\end{rem}
\begin{cor}\label{res-factor}
For all $n\geq 3$, $\bhmv{n}=\ol\im(\ol\iota_n\psi_n)$.
\end{cor}
\begin{proof}
Since $\fq$ is dominant,
$\ol\im(\ol\iota_n\psi_n)=\ol\im(\ol\iota_n\psi_n \fq)
= \ol\im(\ol\iota_n\phi_n)
 =: \bhmv{n}$.
\end{proof}

The unique factorization map $\psi_n^\#$ can be computed directly in Macaulay2 for small $n$. The expressions in moment coordinates are simpler than in probabilities, so we present these in the following proposition.
\begin{prop}\label{res-psi3}
The map $\psi_3^\#$ is given in moment coordinates by 
\begin{align*}
m_\es=m_{000}\mapsto&\; 1\\
m_{1}=m_{100}\mapsto&\; a+u\\
m_{2}=m_{010}\mapsto&\; ab+c+u\\
m_{3}=m_{001}\mapsto&\; ab^2+bc+c+u\\
m_{12}=m_{110}\mapsto&\; abu+ac+au+cu+u^2+bv\\
m_{13}=m_{101}\mapsto&\; ab^2u+abc+bcu+b^2v+ac+au+cu+u^2\\
m_{23}=m_{011}\mapsto&\; ab^2u+abc+abu+bcu+c^2+2cu+u^2+bv\\	
\notag m_{123}=m_{111}\mapsto&\; ab^2u^2+2abcu+abu^2+bcu^2+b^2uv+ac^2+2acu\\
&+c^2u+au^2+2cu^2+u^3+abv+bcv+2buv
\end{align*}
\end{prop}

We will eventually prove the \nr{res-bir} by marginalization to the case $n=3$, which we can prove here:

\begin{prop}{%
\label{res-rec3}
The following triangular set of equations hold on the graph of $\psi_3$, after clearing denominators, and can thus be used to recover parameters from observed moments where the denominators are non-zero:
\begin{align*}
b &= \frac{m_3-m_2}{m_2-m_1}\\
u &= \frac{m_1m_3-m_2^2+m_{23}-m_{12}}{2(m_3-m_2)}\\
a &= m_1 - u\\
c&= a-ba+m_2-m_1\\
v&=a^2-\frac{m_1m_2-m_{12}}{b}
\end{align*}
}\end{prop}
\noindent (This proposition and the following corollary actually hold for all $\phi_n$ with $n\geq 3$, because of \ar{res-marg}, and by \ar{sec-hmms}, these same formulae can be used to recover parameters for $\hmm{2,k,n}$ when $k>2$ as well.)
\begin{proof}
These equations can be checked with direct substitution by hand from \ar{res-psi3}.  Regarding the derivation, they can be obtained in Macaulay2 by computing two Gr\"obner bases of the elimination ideal $I=\langle m_v-\phi_3(m_v)|v\in \labs ^3\rangle$ over the ring $\maff{3}$, in Lex monomial order: once in the ring $\mring{3} [v,c,a,b,u]$, and once in $\mring{3} [v,c,u,b,a]$.  Each variable occurs in the leading term of a some generator in one of these two bases with a simple expression in moments as its leading coefficient.  We solve each such generator (set to $0$) for the desired parameter.
\end{proof}

\begin{cor}
\label{res-rho3}
The map $\psi_3: \CC^5 \to \bhmv{3}$ has a birational inverse $\rho_3$.  The map $\rho_3^\#$  on moment coordinate functions is given by:
\newcommand{\num}{\operatorname{num}}
\begin{align*}
a &\mapsto \frac{m_2^2+m_3m_1-2m_2m_1-m_{23}+m_{12}}{2(m_3-m_2)}
&u & \mapsto \frac{-m_2^2+m_3m_1+m_{23}-m_{12}}{2(m_3-m_2)}\\
b &\mapsto \frac{m_3-m_2}{m_2-m_1}
&v & \mapsto \frac{\num(v)}{4(m_3-m_2)^2}\\
c & \mapsto \frac{\num(c)}{2(m_2-m_1)(m_3-m_2)}, \te{ where}
\end{align*}
\begin{align*}
\num(c) =&-m_1m_2^2+m_1^2m_3+m_2^2m_3-m_1m_3^2-m_1m_{12}\\
 & +2m_2m_{12}-m_3m_{12}+m_1m_{23}-2m_2m_{23}+m_3m_{23}, \te{ and}\\
\num(v)=& \hspace{1ex}m_2^4-2m_1m_2^2m_3+m_1^2m_3^2-2m_2^2m_{12}-2m_1m_3m_{12}+4m_2m_3m_{12}\\
&+4m_1m_2m_{23}-2m_2^2m_{23}-2m_1m_3m_{23}+m_{12}^2-2m_{12}m_{23}+m_{23}^2.
\end{align*}
\end{cor}
\begin{proof}This can be derived by substituting the solutions for $u$, $a$, and $b$ in the previous propositions into the subsequent solutions for $a$, $c$, and $v$.  Alternatively, it can be checked by direct substitution in Macaulay2, \ie, one computes that $\psi_3^\#\circ\rho^\#(\theta)=\theta$ for each birational parameter $\theta\in \{a,b,c,u,v\}$.\end{proof}

The expressions in \ar{res-rho3} are considerably simpler in moment coordinates than in probabilities.  Comparing the number of terms, the numerators for $a,b,c,u,v$ respectively have sizes 5, 2, 10, 4, and 12 in moment coordinates, versus sizes 22, 4, 56, 22, and 190 in probability coordinates.  This explains in part why Macaulay2's Gr\"obner basis computations execute in moment coordinates with much less time and memory.

\subsection{Statistical interpretation of the birational inverse $\rho_3$}\label{sec-intbir}

It turns out that the factors appearing in the denominators of \ar{res-rho3} defining $\rho_3$ have simple factorizations in terms of the rational and birational parameters:

\begin{itemize}{
\item $m_3-m_2$ appears in the denominator of all $\rho_3(\theta)$ except $\rho_3(b)$, and 
$$m_3-m_2 \ \ \stackrel{\psi_3}{\mapsto} \ \ (b)(ab-a+c)  \ \ \stackrel{\fq}{\mapsto} \ \  (b)(v_0)(a_0b-a_0+c_0)$$
\item $m_2-m_1$ appears in the denominator of $\rho_3(b)$ and $\rho_3(c)$, and 
$$m_2-m_1  \ \ \stackrel{\psi_3}{\mapsto} \ \  ab-a+c \ \ \stackrel{\fq}{\mapsto} \ \   (v_0)(a_0b-a_0+c_0)$$
}\end{itemize}
Let us pause to reflect on the meaning of these factors.

\begin{itemize}

\item The factor $v_0$ occurs in $det(E)=2v_0$, hence $v=v_0^2=0$ iff the hidden Markov chain has ``no effect'' on the observed variables.  The image locus $\phi_3(\{v_0=0\})$ can thus be modeled by a sequence of IID coin flips with distribution $E_0=E_1=(1-u,u)$, so the BHMM is an unlikely model choice.  This is a {\bf one-dimensional submodel}, parametrizable by $u\in [0,1]$, with a regular (everywhere-defined) inverse given simply by $u=m_1$.  Denote this model by $\biid{n}$.

\item The factor $b$ occurs in $det(T)=b$, hence $b=0$ iff each hidden node has ``no effect'' on the subsequent hidden nodes.  In this case, 
the observed process can be modeled as a sequence of independent coin flips, the first flip having distribution $(1-\alpha, \alpha) := \pi E$ 
and subsequent flips being IID having distribution $(1-\beta,\beta):=T_0 E = T_1 E$.  The image locus $\phi_3(\{b=0\})$ is hence a {\bf two-dimensional submodel}, parametrizable by $(\alpha,\beta)\in[0,1]^2$, with a regular inverse given by $\alpha=m_1, \beta=m_2$.  Denote this model by $\binid{n}$, for ``binary independent nearly identically distributed'' model, and note that ${\binid{n}}\sps {\biid{n}}$ by setting $\alpha=\beta$.

\item The factor $a_0b-a_0+c_0$ occurs in $\pi T - \pi = \tfrac{1}{2}(-a_0b+a_0-c_0, \ a_0b-a_0+c_0)$.  Hence $a_0b-a_0+c_0=0$ iff $\pi$ is a fixed point of $T$, i.e. 
the hidden Markov chain is {\em at equilibrium}.  We may define the Equilibrium Binary Hidden Markov model by restricting $\phi_n$ to the locus $\{a_0b-a_0+c_0=0\})$, which turns out to yield a {\bf four-dimensional submodel} for each $n\geq 3$.   Denote this submodel by $\ebhmm{n}$.
\end{itemize}

It can be easily shown, with the same methods used here for $\bhmm{n}$, that $\ebhmm{n}$ itself has a birational parametrization by $(a_0v_0,b,u,v_0^2)=(a,b,u,v)$, where $a_0,b\in[-1,1],$  $c_0:=a_0(1-b)\in[|b|-1,1-|b|]$, $v_0\in[0,1],$ and $u\in[|v_0|,1-|v_0|]$, with an inverse parametrization given by 
\begin{align*}
b &= \frac{m_1^2-m_{13}}{m_1^2-m_{12}}
& u &= \frac{2m_1m_{12} - m_1m_{13} - m_{123}}{2(m_1^2-m_{13})}\\
a &= m_1-u
& v &= \frac{a^2b-m_1^2+m_{12}}{b}
\end{align*}

\noindent The newly occurring denominators here are $m_1^2 - m_{12}=(b)(a^2-v)=(b)(v_0)^2(a_0^2-1)$ and $m_1^2-m_{13}=(b)^2(a^2-v)=(b)(v_0)^2(a_0^2-1)$.  It easy to check that the only points of $\ebhmm{n}$ where these expressions vanish are points that lie in $\binid{n}$.  Thus, for $n\geq 3$, $\bhmm{n}$ can be stratified as a union of three statistically meaningful submodels 
\begin{align*}
\bhmm{n} &= {\binid{n}} & \la \textrm{2 dimensional} \\
&\phantom{} \cup \big({\ebhmm{n}} \setminus {\binid{n}} \big) &  \la \textrm{4 dimensional}\\
&\phantom{} \cup \big(\bhmm{n} \setminus \big({\ebhmm{n}} \cup {\binid{n}} \big)\big) &  \la \textrm{5 dimensional}
\end{align*}
each of which has an everywhere-defined inverse parametrization.

\subsection{Computational advantages of moments, cumulants, and birational parameters}
Our approach has been to work with moments $m_v$ and cumulants $k_v$ instead of probabilities $p_v$, and the birational parameters $a,b,c,u,v$ instead of the matrix entries $\pi_1, t_{i1}, e_{i1}$.  Other than the theoretical advantage that the model map is generically injective on the birational parameter space, significant computation gains in Macaulay2 also result from these choices (see \ar{sec-der} for laptop specifications):

\begin{itemize}

\item Computing $\te{ker}\, \psi_3= \te{ker}\, \phi_3$, the affine defining ideal of $\bhmv{3}$, took less than 1 second in Macaulay2 when using the birational parameters, compared to 25 seconds when using the matrix entries and moments, and 15 minutes when using the matrix entries and probabilities.

\item Computing $\te{ker}\, \psi_4= \te{ker}\, \phi_4$, the affine defining ideal of $\bhmv{4}$ took less than 1 second in Macaulay2 when using the birational parameters and {\em cumulant coordinates} \citep{SZ11}, compared to 1.5 hours when using the matrix entries and cumulant coordinates, and running out of memory (8gb) when using the matrix entries and probabilities.
\end{itemize}

\section{Parametrizing BHMMs though a trace variety}\label{sec-trace}
In this section, we exhibit a parametrization of every BHMM through a particular {\em trace variety} called $\spec \tralg$, which itself can be embedded in  $\CC^{10}$.  We use these coordinates to prove the \nr{res-bir} and the \nr{res-sub}, which were stated without proof.  

For this, we will define a map $\phi_\infty$ through which all the $\phi_n$ factor, and using a version of the Baum formula for moments, we factor this map further though $\spec \tralg$.  Then we use a finite set $10$ of generators of the ring $\tralg$ exhibited by \citep{Sib68} to show that the image of  $\phi_\infty$ lands in the desired subring $\nring$, and write $\psi_\infty$ for the factorization. Finally, by marginalizing to the case $n=3$, we obtain a birational inverse for $\psi_n$ from the map $\rho_3$ given in \ar{res-rho3}.

\subsection{Marginalization maps}

For each pair of integers $n' \geq n \geq 1$, the {\em marginalization map} $\mu^{n'}_{n} : \paff{n'} \to \paff{{n}}$ is given by
$${\mu^{n'}_{n}}^\# (p_v) := \sum_{|w|=n'-{n}} p_{vw}$$
These restrict to maps $\mu^{n'}_{n}: \spaff{n'}\to\spaff{{n}}$, and define {\em rational} maps $\mu^{n'}_{n}: \ppro{n'} \dra \ppro{{n}}$.
In moment coordinates, these maps are actually coordinate projections: 
${\mu^{n'}_{n}}^\# (m_v) = m_{v\ol 0}$
where $\ol 0$ denotes a sequence of ${n'-n}$ zeros.  In fact, using the subset notation for moments $m_I$, the corresponding ring maps are literal inclusions: ${\mu^{n'}_{n}}^\# (m_I) = m_{I}$.  In other words, $\mu^{n'}_{n}: \maff{n'} \to \maff{{n}}$ is just the map which forgets those $m_I$ where $I\nsubseteq [n]$.

\subsection{The Baum formula for moments}\label{sec-baum}

Equation \eqref{eqn-bhmmap} involves $O(2^n)$ addition operations.  There is a faster way to compute $\phi_n\us(p_v)$, using $O(n)$ arithmetic operations, by treating the BHM process as a {\em finitary process} \citep{Sch11}.  We define two new matrices\footnote{$P$ can be thought of naturally as a $2\times2\times2$ tensor, but we will not make use of this interpretation.}
\begin{align*}
(P_i)_{jk} &:= E_{ji}T_{jk} = \pr(V_t=i\te{ and } H_{t+1}=k\, | \,H_t=j \te{ and } \pi, E,T), \te{ that is,}\\[2ex]
P_0 &:= \bmat{
T_{00}E_{00} & T_{01}E_{00}  \\ 
T_{10}E_{10} & T_{11}E_{10}} \quad \te{and} \quad 
P_1 := \bmat{
T_{00}E_{01} & T_{01}E_{01}  \\ 
T_{10}E_{11} & T_{11}E_{11}}
\end{align*}
Writing $\one $ for the vector $(\smat{1\\ 1})$ we obtain the matrix expression 
$\phi\us(p_v) = \pi P_{v_1}P_{v_2}\cdots P_{v_v}\one$
which involves only $4n +2$ multiplications  and $2n+1$ additions.   This is known as the Baum formula.  We can rewrite this formula as a trace product of $2\times 2$ matrices:
\begin{align*}
\phi\us(p_v) &= \tr(\pi P_{v_1}P_{v_2}\cdots P_{v_n}\one ) = \tr((\one \pi) P_{v_1}P_{v_2}\cdots P_{v_n} )
\end{align*}
To create an analogue of this formula in moment coordinates, we let
\begin{align*}
M_0 &:= P_0+P_1 = T & M_1 &:= P_1 & M_2 := \one \pi = \bmat{\pi_0 & \pi_1 \\ \pi_0 & \pi_1}
\end{align*}
\begin{prop}[Baum formula for moments]\label{res-baum}The binary hidden Markov map $\phi_n$ can be written in moment coordinates as
$$\phi_n^\# (m_v) = \tr(M_2 M_{v_1}M_{v_2}\cdots M_{v_n})$$
\end{prop}
\noindent For example, $\phi_n^\# (m_{01001}) = \tr(M_2 M_{0}M_{1}M_{0}M_{0}M_{1})$.
\begin{proof}
By our definition of $m_v$ \eqref{eqn-mom}, we have 
\begin{align*}
\phi_n\us (m_v) &=  \sum_{w\geq v} \phi_n\us (p_w) =\sum_{w\geq v} \tr((\one \pi) P_{w_1}P_{w_2}\cdots P_{w_n} )\\
&=\tr\left((\one \pi)
\left(\sum_{w_1\geq v_1}P_{w_1}\right)
\left(\sum_{w_2\geq v_2} P_{w_2}\right)
\cdots 
\left(\sum_{w_n\geq v_n} P_{w_n}\right) 
\right)\\[2ex]
&= \tr(M_2 M_{v_1}M_{v_2}\cdots M_{v_n}) \qedhere
\end{align*}
\end{proof}
\subsection{Truncation and $\phi_\infty$}\label{sec-trunc}

\begin{prop}\label{res-marg}The binary hidden Markov maps $\phi_n$ form a directed system of maps under marginalization, meaning that, for each $n'\geq n \geq 1$, the following diagrams commute:
\begin{center}
\begin{tikzpicture}[node distance=3cm, auto, semithick, >=stealth']
  \node (Ct) {$\taff$};
  \node (X) [right of=Ct, node distance=2cm]{};
  \node (Vn) [below of=X, node distance=1cm] {$\smaff{n}$};
  \node (Vnk) [above of=X, node distance=1cm] {$\smaff{n'}$};
  \draw[->] (Ct) to node {$\phi_{n'}$} 	(Vnk);
  \draw[->] (Ct) to node [swap] {$\phi_{n}$}	(Vn);
  \draw[->] (Vnk) to node {$\mu^{n'}_{n}$}	(Vn);
  \node (Rt) [right of=Ct, node distance=6cm] {$\tring$};
  \node (X2) [right of=Rt, node distance=2cm]{};
  \node (Rn) [below of=X2, node distance=1cm] {$\smring{n}$};
  \node (Rnk) [above of=X2, node distance=1cm] {$\smring{n'}$};
  \draw[<-] (Rt) to node {$\phi_{n'}^\#$} 	(Rnk);
  \draw[<-] (Rt) to node [swap] {$\phi_{n}^\#$}	(Rn);
  \draw[<-] (Rnk) to node {${\mu^{n'}_{n}}^\#$}	(Rn);
\end{tikzpicture}
\end{center}
\end{prop}
\begin{proof}
This can be seen directly from the definition of $\phi_n$ using \eqref{eqn-bhmmap} and of $m_v$ in \eqref{eqn-mom}.  Alternatively, observe that because $M_0=T$ is stochastic, $M_0M_2 = M_2$, so for any sequence $\ol 0$ of length $n'-n$, the \nameref{res-baum} (\ar{res-baum}) implies that
\begin{equation}\label{eqn-trunc}
\phi_{n'}^\# (m_{v\ol 0}) = \phi_n^\# (m_v)\qedhere
\end{equation}
\end{proof}
Thus, to compute $\phi_n$ for all $n$, it is only necessary to compute those $\phi_n\us m_{v'}$ where $v'$ ends in $1$.  Motivated by this observation, let
$\smring{\infty} := \CC[m_{v1}\; | \; v\in \labs^n \te{ for some } n\geq 0]
=\CC[m_1, m_{01}, m_{11}, m_{001}, m_{101}, m_{011},\ldots]
$,
which in subset index notation is simply
\begin{align*}
\smring{\infty} &:= \CC[m_I\; | \; I\sbs [n] \te{ for some } n\geq 0]\\
&=\CC[m_1, m_{2}, m_{12}, m_{3}, m_{13}, m_{23},\ldots]
\end{align*}
Then we define
$\phi_\infty: \taff \to \spec \smring{\infty}$ and 
$\phi_\infty\us: \tring \la \smring{\infty}$
by the formula 
$\phi_\infty\us(m_{v1\ol 0}) := \phi\us_{\te{length}(v1)}(m_{v1})$, i.e.
\begin{align}\label{eqn-phiinf}
\phi_\infty\us(m_{I}) &:= \phi\us_{\te{size}(I)}(m_I)
\end{align}
Note that by locating the position of the last $1$ in a binary sequence $v'\neq 0\ldots 0$, we can write $v'$ in the form $v1\ol 0$ for a unique string $v$ (possibly empty if $v'=1$), so this map is  well-defined.  By the same principle, for each $n$ we can also define a ``truncation'' map 
$\trunc: \; \spec \smring{\infty} \to \smaff{n}$ by 
$\trunc\us (m_{v1\ol 0}) := m_{v1}$, 
which, in subset index notation, is a literal ring inclusion:
\begin{align}
\label{eqn-trunc} \trunc\us (m_I) &:= m_I
\end{align}

With this definition, $\phi_n\us$ factorizes as 
$\phi_n\us = \phi_\infty\us \circ \trunc_n\us$.  
We can summarize this and \ar{res-marg} as follows:

\begin{prop}\label{res-dir}For all $n'\geq n \geq 1$, the following diagrams commute:
\begin{center}
\begin{tikzpicture}[node distance=2cm, auto, semithick, >=stealth']
  \node (Co){$\taff$};
  \node (X2) [below of =Ct, node distance=2cm]{};
  \node (Cn) [left of =X2, node distance=2.5cm] {$\smaff{n}$};
  \node (Cnk) [right of =X2, node distance=0cm] {$\smaff{n'}$};
  \node (Cinf) [right of =Cnk, node distance=2.5cm] {$\spec \smring{\infty}$};
  \draw[->] (Co) to node [swap] {$\phi_{n}$}	(Cn);
  \draw[->] (Co) to node {$\phi_{n'}$} 	(Cnk);
  \draw[->] (Co) to node {$\phi_{\infty}$} 	(Cinf);
  \draw[<-] (Cn) to node  [swap] {${\mu^{n'}_n}$}	(Cnk);
  \draw[<-] (Cnk) to node  [swap] {$\trunc_{n'}$}	(Cinf);
  \node (Ro) [right of = Co, node distance=8cm]{$\tring$};
  \node (X) [below of =Ro, node distance=2cm]{};
  \node (Rn) [left of =X, node distance=2.5cm] {$\smring{n}$};
  \node (Rnk) [right of =X, node distance=0cm] {$\smring{n'}$};
  \node (Rinf) [right of =Rnk, node distance=2.5cm] {$\smring{\infty}$};
  \draw[<-] (Ro) to node [swap] {$\phi_n\us$}	(Rn);
  \draw[<-] (Ro) to node {$\phi_{n'}\us$} 	(Rnk);
  \draw[<-] (Ro) to node {$\phi_{\infty}\us$} 	(Rinf);
  \draw[->] (Rn) to node  [swap] {${\mu^{n'}_n}^\#$}	(Rnk);
  \draw[->] (Rnk) to node  [swap] {$\trunc\us_{n'}$}	(Rinf);
  \end{tikzpicture}
\end{center}
\end{prop}
\begin{rem}\label{rem-col}
These diagrams exhibit the rings $\smring{n}$ and maps $\phi_n\us$ as a directed system under the inclusion maps ${\mu^{n'}_{n} }\us$, such that $\smring{\infty} = \operatorname{colim}_{n\to \infty} \smring{n}$ and 
$\phi_{\infty}^\# = \operatorname{lim}_{n\to \infty} \phi^\#_n$.
\end{rem}

Now, to prove that $\phi_n$ factors through $\fq$, we need only show that $\phi_\infty$ does.

\subsection{Factoring $\phi_\infty$ through a trace variety}\label{sec-birprooffact}
Let $X_0,X_1,X_2$ be $2\times 2$ matrices of indeterminates,
\begin{equation*}
X_0 = \bmat{
x_{000} & x_{001}  \\ 
x_{010} & x_{011}} \qquad 
X_1 = \bmat{
x_{100} & x_{101}  \\ 
x_{110} & x_{111}} \qquad 
X_2 = \bmat{
x_{200} & x_{201}  \\ 
x_{210} & x_{211}}
\end{equation*}
and following the notation of \citep{Dre05},
$\xalg := \CC[\te{entries of }X_0,X_1,X_2]$
denotes the polynomial ring on the entries $x_{ijk}$ of these three $2\times 2$ matrices.  The {\em trace algebra} $\tralg$ is defined as the subring of $\xalg$ generated by the traces of products of these matrices,
$\tralg :=\CC[\tr(X_{i_1}X_{i_2}\cdots X_{i_r}) \; | \; r \geq 1] \sbs \xalg$
and we refer to $\spec\tralg$ as a {\em trace variety}.  We write 
\begin{equation*}
\nu : \spec \xalg \to \spec \tralg \qquad \te{ and } \qquad
\nu\us : \tralg \hra \xalg
\end{equation*}
for the natural dominant map and corresponding ring inclusion.  To relate these varieties to binary HMMs , we define two new maps
$\omega\us :  \; \xalg \to \tring$ and $\xi\us : \; \smring{\infty} \to \tralg$ by
$$
\omega\us(X_i) := M_i \qquad \textrm{and} \qquad \xi\us (m_{v1}) := \tr\left(\left(X_2\prod_{i \in v} X_i\right)X_1\right).
$$
\begin{prop}[Baum factorization]\label{res-baumfact}The ring map $\phi_\infty\us$ factorizes as $\phi_\infty\us = \omega\us \circ \nu\us \circ \xi \us,$
i.e., the following diagram commutes:
\begin{center}
\begin{tikzpicture}[node distance=2cm, auto, semithick, >=stealth']
  \node (Co){$\taff$};
  \node (SpO) [below of =Co, node distance=2cm]{$\spec\xalg$};
  \node (SpC) [right of=SpO, node distance=3cm] {$\spec\tralg$};
  \node (Cinf) [right of =Co, node distance=3cm] {$\spec \smring{\infty}$};
  \draw[->] (Co) to node  [swap] {$\omega$}	(SpO);
  \draw[->] (SpO) to node [swap]  {$\nu$}	(SpC);
  \draw[->] (SpC) to node [swap]  {$\xi$}	(Cinf);
  \draw[->] (Co) to node  {$\phi_{\infty}$}	(Cinf);
  \end{tikzpicture}
\end{center}
\end{prop}
\begin{proof}This is just a restatement of the \nameref{res-baum} (\ar{res-baum}): 
\begin{align*}
\omega\us (\nu\us(\xi\us (m_{v1}))) 
& =  \omega\us \tr\left(X_2\prod_{i \in v1} X_i\right) = \tr\left(M_2\prod_{i \in v1} M_i\right)\\
& = \phi_{length(v1)}\us (m_{v1}) = \phi_\infty\us (m_{v1}) \qedhere
\end{align*}
\end{proof}

\subsection{Proving the \nr{res-sub}}\label{sec-sub}
We begin by seeking a factorization of the map $\omega\us\circ \nu\us$.  For this we apply the following commutative algebra result of Sibirskii on the trace algebras $C_{2,r}$:

\begin{prop}[Sibirskii, 1968] The trace algebra $C_{2,r}$ is generated by the elements
\begin{align*}
\tr(X_i) \; : \; & 0\leq i \leq r\\
\tr(X_iX_j) \; : \; & 0\leq i \leq j \leq r\\
\tr(X_iX_jX_k) \; : \; & 0\leq i < j < k \leq r
\end{align*}
\end{prop}

\begin{cor}\label{res-trgens}The algebra $C_{2,3}$ is generated by the 10 elements
\begin{align*}
&\tr(X_0), \tr(X_1), \tr(X_2), \\
&\tr(X_0^2), \tr(X_1^2), \tr(X_2^2), \tr(X_0X_1), \tr(X_0X_2),  \tr(X_1X_2),  \\
&\tr(X_0X_1X_2)
\end{align*}
\end{cor}

\begin{prop}\label{res-trfact}
The ring map $\omega\us\circ \nu\us$
factors through the inclusion $$\fq\us \; : \; \nring:=\CC[a,b,c,u,v] \; \hra \; \tring := \CC[a_0,b,c_0,u,v_0],$$
i.e. we can write $\omega\us\circ \nu\us=\fq\us\circ \fr\us$ so that the following diagram commutes: 
\begin{center}
\begin{tikzpicture}[node distance=2cm, auto, semithick, >=stealth']
  \node (Co) {$\taff$};
  \node (SpO) [below of =Co, node distance=2cm]{$\spec\xalg$};
  \node (SpC) [right of=SpO, node distance=3cm] {$\spec\tralg$};
  \node (Cn) [right of =Co, node distance=3cm] {$\naff$};
  \draw[->] (Co) to node [swap] {$\omega$}	(SpO);
  \draw[->] (SpO) to node [swap] {$\nu$}	(SpC);
  \draw[->] (Cn) to node  {$\fr$}	(SpC);
  \draw[->] (Co) to node {$\fq$}	(Cn);
  \end{tikzpicture}
\end{center}
\end{prop}
\begin{proof}
We apply $\omega\us$ to the ten generators of $\tralg$ given in \ar{res-trgens} and check that they land in $\nring$.  Explicit, we find that:
\begin{align*}
&\tr(M_0) = b+1 && \tr(M_1) = bu+c+u && \tr(M_2) =1\\
&\tr(M_0^2) = b^2+1 && \tr(M_1^2) = b^2u^2+2bcu \rlap{$+c^2+2cu+u^2+2bv$}\\
&\tr(M_2^2) = 1 && \tr(M_0M_1) = b^2u+bc+c+u && \tr(M_0M_2) = 1\\
&\tr(M_1M_2) = a+u && \tr(M_0M_1M_2) = ab+c+u\qedhere
\end{align*}
\end{proof}
Now, by letting $\psi_\infty\us:=\fr\us\circ\xi\us$ we may factor the ring map $\phi_\infty\us$ as 
$$\phi_\infty\us = \omega\us\circ \nu\us\circ \xi\us = \fq\us\circ \fr\us \circ\xi\us = \fq\us\circ \psi_\infty\us.$$

\begin{cor}\label{res-psibarinf} The following diagram commutes:
\begin{center}
\begin{tikzpicture}[node distance=3cm, auto, semithick, >=stealth']
  \node (Ceta) {$\naff$};
  \node (Co) [left of=Ceta, node distance=3cm] {$\taff$};
  \node (X) [below of =Ceta, node distance=2.5cm]{};
  \node (SpO) [below of =Co, node distance=2cm]{$\spec\xalg$};
  \node (SpC) [right of=SpO, node distance=3cm] {$\spec\tralg$};
  \node (Cinf) [right of =Ceta, node distance=3cm] {$\spec\smring{\infty}$};
  \draw[<-] (Ceta) to node [swap] {$\fq$}	(Co);
  \draw[->] (Ceta) to node [swap] {$\fr$}	(SpC);
  \draw[->] (Ceta) to node {$\psi_{\infty}$} 	(Cinf);
  \draw[->, bend left] (Co) to node {$\phi_{\infty}$} 	(Cinf);
  \draw[->] (Co) to node  [swap] {$\omega$}	(SpO);
  \draw[->] (SpO) to node [swap]  {$\nu$}	(SpC);
  \draw[->] (SpC) to node [swap]  {$\xi$}	(Cinf);
\end{tikzpicture}
\end{center}
\end{cor}
\begin{proof}[Proof of the \nr{res-sub}]\label{proof-sub} \ar{res-dir} and \ar{res-psibarinf} together imply that the following diagrams commute:
\begin{center}
\begin{tikzpicture}[node distance=3cm, auto, semithick, >=stealth']
  \node (Ceta) {$\naff$};
  \node (Co) [left of=Ceta, node distance=2.5cm] {$\taff$};
  \node (Cinf) [right of =Ceta, node distance=2.5cm] 
{$\spec\smring{\infty}$};
  \node (Cn) [right of=Cinf, node distance=2.5cm] {$\smaff{n}$};
  \draw[->] (Co) to node {$\fq$}	(Ceta);
  \draw[->] (Ceta) to node {$\psi_{\infty}$} 	(Cinf);
  \draw[->] (Cinf) to node {$\trunc_n$} 	(Cn);
  \draw[->, bend right] (Co) to node [swap] {$\phi_n$} 	(Cn);
  \node (Reta) [below of=Ceta, node distance=3cm]{$\nring$};
  \node (Ro) [left of=Reta, node distance=2.5cm] {$\tring$};
  \node (Rinf) [right of =Reta, node distance=2.5cm] 
{$\smring{\infty}$};
  \node (Rn) [right of=Rinf, node distance=2.5cm] {$\smring{n}$};
  \draw[<-] (Ro) to node {$\fq\us$}	(Reta);
  \draw[<-] (Reta) to node {$\psi_{\infty}\us$} 	(Rinf);
  \draw[<-] (Rinf) to node {$\trunc_n\us$} 	(Rn);
  \draw[<-, bend right] (Ro) to node [swap] {$\phi_n\us$} 	(Rn);
 \end{tikzpicture}
\end{center}
In particular, the map $\phi_n\us$ factors through $\nring$, as required. 
\end{proof}

\subsection{Proving the \nr{res-bir}}\label{sec-birproof}
Recall that \ar{res-dom} implies the first part of the \nr{res-bir}, by taking $$\fq : \taff \lra \naff$$ as the generically $2:1$ map.  Thus, it remains to show that the maps 
$$\psi_n: \naff \lra \bhmv{n}$$ 
have birational inverses $\rho_n$.  The inverse map $\rho_3$ was already exhibited in \ar{res-rho3}, and we obtain $\rho_n$ by marginalization: let
$$\rho_n = \rho_3 \circ \mu^n_3.$$
Let $U\sbs\naff$ be the Zariski open set on which $\psi_3$ is an isomorphism with inverse $\rho_3$.  Consider the set $\psi_n(U) \sbs \bhmv{n}$.  It is Zariski dense in $\bhmv{n}$, and by Chevalley's theorem (\citetalias{EGA4}, 1.8.4), it is constructible, so it must contain a dense open set $W'\sbs\bhmv{n}$.  Now let $W=\psi_n\-(W')$, so we have $\psi_n(W)=W'\sbs \psi_n(U)$.  
\begin{prop}$\rho_n \circ \psi_n = \te{Id}$ on $W$ and $\psi_n \circ \rho_n = \te{Id}$ on $W'$.
\end{prop}
\begin{proof}
Suppose $\wh\eta \in W$.  Then $\rho_n \circ \psi_n (\wh \eta)
=\rho_3 \circ \mu^n_3 \circ \psi_n(\wh\eta) 
=\rho_3 \circ \psi_3(\wh\eta)
= \wh \eta$ since $\wh\eta\in U$.  
Now suppose $\wh p \in W'$, so $\wh p = \psi_n(\wh\eta)$ 
for some $\wh \eta \ \in W$.  Then, applying \ar{res-marg},
\begin{align*}
\psi_n \circ \rho_n (\wh p) 
&= \psi_n \circ \rho_n \circ \psi_n(\wh\eta) = \psi_n \circ \rho_3 \circ \mu^n_3 \circ \psi_n(\wh\eta)\\ 
&= \psi_n \circ \rho_3 \circ \psi_3(\wh\eta)= \psi_n (\wh\eta) = \wh p \qedhere\\ 
\end{align*}
\end{proof}
This completes the proof of the \nr{res-bir}.  In fact we have also proven the following:

\begin{thm}\label{res-tri}For any $n'\geq n\geq 3$, there is a commutative diagram of dominant maps: 
\end{thm}

\begin{center}
\begin{tikzpicture}[node distance=2cm, auto, semithick, >=stealth']
  \node (Cn) {$\CC_\eta^5$};
  \node (X) [right of=Cn, node distance=3cm]{};
  \node (Vn) [below of=X, node distance=1cm] {$\bhmv{n}$};
  \node (Vn1) [above of=X, node distance=1cm] {$\bhmv{n'}$};
  \draw[->] (Cn) to node {$\psi_{n'}$} 	(Vn1);
  \draw[->] (Cn) to node [swap] {$\psi_n$}	(Vn);
  \draw[->, dashed] (Vn1) to node {$\mu^{n'}_n$}	(Vn);
  \node (Co) [left of=Cn] {$\taff$};
  \draw[->, bend right] (Co) to node [swap] {$\phi_n$}	(Vn);
  \draw[->, bend left] (Co) to node {$\phi_{n'}$}	(Vn1);
  \draw[->] (Co) to node {$\fq$}		(Cn);
\end{tikzpicture}
\end{center}

\section{Applications and future directions}\label{sec-app}

Besides attempting to compute a set of generators for $\bhmi{5}$, there are many other questions to be answered about HMMs that can be approached immediately with the techniques of this paper.

\subsection{A nonnegative distribution in $\bhmv{3}$ but not $\bhmm{3}$}

It turns out that not all of the probability distributions (non-negative real points) of $\bhmv{n}$ lie in the model $\bhmm{n}$.  In other words, $\bhmv{n}\cap \psim{n} \neq \bhmm{n}$, so the model must be cut out by some non-trivial inequalities inside the simplex.  To illustrate this, the following real point $\wh \theta$ of $\taff$ does not lie in $\tsim$, but maps under $\phi_3$ to a point $\wh{p}$ of $\Delta_p^{7}$:

\begin{equation}\label{eqn-neg}
\wh\theta = (\wh\pi, \wh T, \wh E) = \left(\bmat{-\tfrac{1}{8} & \tfrac{9}{8}}, \quad 
\bmat{
\tfrac{3}{4} & \tfrac{1}{4}  \vspace{1ex}\\ 
\tfrac{1}{4} & \tfrac{3}{4}}, \quad 
\bmat{
\tfrac{3}{4} & \tfrac{1}{4}  \vspace{1ex}\\ 
\tfrac{1}{4} & \tfrac{3}{4}}\right)
\end{equation}

Moreover, the analysis of \ar{sec-intbir} reveals that the fiber $\phi_3\-(\wh p)$ consists only of the point $\wh \theta$ and the ``swapped'' point
\begin{equation}\label{eqn-negswap}
\wh\theta' = (\wh\pi', \wh T', \wh E') = \left(\bmat{\tfrac{9}{8} & -\tfrac{1}{8}}, \quad 
\bmat{
\tfrac{3}{4} & \tfrac{1}{4}  \vspace{1ex}\\ 
\tfrac{1}{4} & \tfrac{3}{4}}, \quad 
\bmat{
\tfrac{1}{4} & \tfrac{3}{4}  \vspace{1ex}\\ 
\tfrac{3}{4} & \tfrac{1}{4}}\right)
\end{equation}
which is also not in $\tsim$.  Hence the image point $\wh p=\phi_3(\wh\theta)=\phi_3(\wh\theta')$ is a non-negative point of $\bhmv{3}$ that does not lie in $\bhmm{3}$.

\subsection{A semialgebraic model membership test}
In light of the fact that not every nonnegative distribution in $\bhmv{n}$ is in $\bhmm{n}$, the defining equations of $\bhmv{n}$ are not sufficient to test a probability distribution for membership to the model.  
Using the method of \ar{sec-hmms}, membership to $\hmm{2,k,n}$ can be tested by reducing to the $k=2$ to recover the parameters.

So, suppose we are given a distribution $p\in \psim{n}$ and asked to determine whether $p\in \bhmm{n}$.  The following procedure yields either 
\begin{itemize}

\item[(1)] a proof by contradiction that $p \notin \bhmm{n}$,

\item[(2)] a parameter vector $\theta\in \tsim$ such that $\phi_n(\theta) = p\in\bhmm{n}$, or

\item[(3)] a reduction of the question to whether $p$ lies in one of the lower-dimensional submodels of $\bhmm{n}$ discussed in \ar{sec-intbir}.
\end{itemize}
How to proceed from (3) is essentially the same as what follows, using the birational parametrizations of the respective submodels given in \ar{sec-intbir}.

To begin, we let $p'=\mu^n_3(p)\in\psim{3}$, i.e. we marginalize $p$ to the distribution $p'$ it induces on the first three visible nodes.  Note that if $p\in\bhmm{n}$ then $p'\in\bhmm{3}$.  Observing the moments $m_I$ of $p'$, if any denominators in the formulae of \ar{res-rho3} vanish, then we end in case (3).  

Otherwise, we let $(a,b,c,u,v)=\psi_3\-(p')$, choose $v_0$ to be either square root of $v$, and let $a_0=a/v_0$, $c_0=c/v_0$.  
If $p$ were due to some BHM process, then by \ar{res-tri}, these would be its parameters, up to a simultaneous sign change of $(a_0,b_0,v_0)$.  
With this in mind, we define $\theta=(\pi,T,E)$ using \eqref{eqn-pite}.  
If $(\pi,T,E)$ are not non-negative stochastic matrices, then $p\notin\bhmm{n}$ and we end in case (1).  
If they are, we compute $p''=\phi_n(\theta)$, and if $p=p''$ then we end in case (2).  
Otherwise $p$ must not have been in $\bhmm{n}$, so we end in case (1).

Note that since all the criteria in this test are algebraic equalities and inequalities, this procedure implicitly describes a semialgebraic characterization of $\bhmm{n}$ for all $n\geq 3$.

\subsection{Identifiability of parameters}
By a {\em rational map} on a possibly non-algebraic subset $\Theta\sbs \CC^k$, we mean any rational map on the Zariski closure of $\Theta$, which will necessarily be defined as a function on a Zariski dense open subset of $\Theta$.  We define polynomial maps on $\Theta$ similarly.

Let $\phi: \Theta \to \CC^n$ be an algebraic statistical model, where as usual we assume $\Theta\sbs\CC^k$ is Zariski dense, and therefore Zariski irreducible.  A (rational) parameter of the model is any rational map $s: \Theta \to \CC$.  Such parameters form a field, $K\simeq\ffield(\CC^k)$.  In applications such as \citep*{MED09}, it is important to know to what extent a parameter can be identified from observational data alone.  In other words, given $\phi(\theta)$, what can we say about $s(\theta)$?  This leads to several different notions of parameter identifiability, as discussed by \citet*{GSS10}.

\begin{defn}We say that a rational parameter $s\in K$ is 
\begin{itemize}
\item {\em (set-theoretically) identifiable} if $s=\sigma \circ \phi$ for some set-theoretic function $\sigma:\phi(\Theta) \to \CC$.  In other words, for all $\theta,\theta'\in\Theta$, if $\phi(\theta)=\phi(\theta')$ then $s(\theta)=s(\theta')$.

\item {\em rationally identifiable} if $s=\sigma \circ \phi$ for some {\em rational} map $\sigma:\phi(\Theta) \to \CC$ (this notion is used without a name by \citet{GSS10}).

\item {\em generically identifiable} if there is a (relatively) Zariski dense open subset $U\sbs\Theta$ such that $s|_U=\sigma \circ \phi|_U$ for some set-theoretic function $\sigma:\phi(U) \to \CC$.

\item {\em algebraically identifiable} if there is a polynomial function $g(p,q):=\sum_i g_i(p_1,\ldots,p_n)q^i$ on $\phi(\Theta) \times \CC$ of degree $d > 0$ in $q$ (so that $g_d$ is not identically $0$ on $\phi(\Theta)$) such that $g(\phi(\theta),s(\theta))=0$ for all $\theta\in\Theta$ (and hence all $\theta \in \CC^k$).  
\end{itemize}
\end{defn}

\begin{que}\label{que-idcom}
What combinations of BHM parameters are rationally identifiable, generically identifiable, or algebraically identifiable?
\end{que}

To answer this question we introduce a lemma on algebraic statistical models in general:

\begin{lem} For any algebraic statistical model $\phi$ as above, the sets $K_{ri}$, $K_{gi}$, and $K_{ai}$, of rationally, generically, and algebraically identifiable parameters, respectively, are all fields.
\end{lem}
\begin{proof}
Since $\Theta$ is Zariski irreducible, so is $\phi(\Theta)$.  Hence the set of rational maps on $\phi(\Theta)$ is simply the fraction field of its Zariski closure (an irreducible variety), and $K_{ri}$ is the image of this field under $\phi^\#$, which must be a field.

For $K_{gi}$, the crux is to show that if $s,s'\in K_{gi}$ and $s\neq 0$ then $s'/s \in K_{gi}$.  
Let $U\sbs\Theta$ and $\sigma:\phi(U) \to \CC$ be as in the definition for $s$, and likewise $U'\sbs \Theta$ and $\sigma:\phi(U') \to \CC$ for $s'$.  
Let $U''=\{\theta\in U\cap U' \; | \; s(\theta) \neq 0\}$, which, being an intersection of three Zariski dense open subsets of $\Theta$, is a dense open.  
We have $\sigma\neq 0$ on $\phi(U'')\sbs \phi(U)\cap \phi(U')$, so we can let $\sigma''=\sigma'/\sigma : \phi(U'') \to \CC$, and then $\sigma'' \circ \phi = s'/s$, so $s'/s\in K_{gi}$. 
Thus $K_{gi}$ is stable under division, and simpler arguments show it is stable stable under $+,-$, and $\cdot$, so it is a field.

Finally, $K_{ai}$ is expressly the relative algebraic closure in $K$ of the image under $\phi^\#$ of the coordinate ring of $\phi(\Theta)$, which is therefore a field.
\end{proof}

\begin{prop}For any algebraic statistical model $\phi$ as above, $K_{ri}\sbs K_{gi} \sbs K_{ai} \sbs K$.
\end{prop}
\begin{proof}
This is now just a restatement of Proposition 3 in \citep{GSS10}.
\end{proof}

Now, the answer to our identifiability question for BHM parameters can be given easily in the coordinates of \ar{sec-bir}.  Here $\phi$ is the BHM map $\phi_n$.  The field $K_{ri}$ is simply the image $\fq^\#(\ffield(\naff))$ because by \ar{res-bir}, $$\psi^\# : \ffield(\bhmv{n})\to \ffield(\naff)$$ is an isomorphism.  
Hence the rationally identifiable parameters are precisely the field of rational functions in $(a,b,c,u,v)=(a_0v_0,b,c_0v_0,u,v_0^2)$ (see \eqref{eqn-pite} for the meanings of these parameters).  
Since $K$ is a quadratic field extension of $K_{ri}$ given by adjoining $v_0 = \sqrt{v}$, and $K_{ai}$ is the algebraic closure of $K_{ri}$ in $K$ (almost by definition), it follows that $K_{ai}=K$, i.e. {\em all parameters} are algebraically identifiable.  
Finally, we observe that, by the action of $\sw$ in \ar{sec-linpar}, there are generically two possible values of $v_0=\frac{1}{2}(E_{11}-E_{01})$ for a given observed distribution, namely $\pm \sqrt{v}$.  
Hence $v_0\notin K_{gi}$, and since a quadratic field extension has no intermediate extensions, it follows that $K_{ri}=K_{gi}$, i.e. all generically identifiable parameters are in fact rationally identifiable.  In summary,
\begin{prop}For $\bhmm{n}$ where $n\geq 3$,
$$\CC(a,b,c,u,v) = K_{ri} = K_{gi} \subsetneq K_{ai} = \CC(a_0,b,c_0,u,v_0)$$
\end{prop}

\subsection{A new grading on BHMM invariants}
The re-parametrized model map $\psi_n$ is homogeneous in cumulant and moment coordinates, with respect to a $\ZZ$-grading where $deg(m_v)=deg(k_v)=sum(v)$, $deg(b)=0$, $deg(a)=deg(c)=deg(u)=1$, and $deg(v)=2$.  This grading allows for fast linear algebra techniques that solve for low degree model invariants as in \citep{BM05}, except that this grading is intrinsic to the model.  Bray and Morton's grading, which is in {\em probability} coordinates, is not on the binary HMM proper, but on a larger variety obtained by relaxing the parameter constraints that the transition and emission matrix row sums are $1$.  The invariants obtained in their search are hence invariants of this larger variety, and exclude some invariants of $\bhmm{n}$.  The grading presented here can thus be used to complete their search for invariants up to any finite degree.

\subsection{Equilibrium BHM processes}

In \ar{sec-intbir} we found that if a BHM process is at equilibrium, our formula for $\psi_3\-$ is undefined.  We may define Equilibrium Binary Hidden Markov Models, EBHMMs, by restricting $\phi_n$ to the locus $\{a_0b-a_0+c_0=0\}$, which turns out to yield a four-dimensional submodel of $\bhmm{n}$ for each $n\geq 3$.  The same techniques used here to study BHMMs have revealed that the EBHMMs, too, have birational parametrizations, and the ideal of $\ebhmm{3}$ is generated by the equations $m_1=m_2=m_3$ and $m_{12}=m_{13}$.  The geometry of EBHMMs will need to be considered explicitly in future work to identify the learning coefficients of BHMM fibers.

\subsection{Larger hidden Markov models}
As we have remarked throughout, many results on $\bhmm{n}$ can be readily applied to $\hmm{2,k,n}$, i.e. HMMs with two hidden states and $k$ visible states $\alpha_1,\ldots ,\alpha_k$.  For example, consider the parameter identification problem.  We may specify the process by a $2\times k$ matrix $E$ of emission probabilities, along with a triple $(a_0,b,c_0)$ defining the $\pi$ and $T$ of the two-state hidden Markov chain as in \eqref{eqn-pite}.  As in \ar{sec-hmms}, to obtain $E_{0\ell}$ and $E_{1\ell}$ from the observed probability distribution for any fixed $j$, we simply define a BHM process by letting $\alpha_\ell=1$ and $\alpha_j=0$ for $j\neq \ell$.  Applying \ar{res-rec3} to the moments of the distribution yields values for $(a,b,c,u,v)$ provided the genericity condition that the denominators involved do not vanish.  Letting $v_0=\sqrt{v}$, $a_0=a/v_0$, and $c_0=c/v_0$, we obtain $(a_0,b,c_0,u,v_0)$ up to a simultaneous sign change on $(a_0,c_0,v_0)$ corresponding to swapping the hidden alphabet as in \ar{sec-linpar}.  Then $E_{0\ell} = u-v$ and $E_{1\ell}=u+v$, and we get $\pi,T$ as well from $(a_0,b,c_0)$.  We can repeat this for each $\ell=1,\ldots,k$ to obtain all the emission parameters, and hence identify all the process parameters modulo the swapping operation.  

For each $\ell \in \{1,\ldots,k\}$, we can also obtain many polynomial invariants of $\hmm{2,n,k}$ by reducing to $\bhmm{n}$ as above, and marginalizing to collections of 4 equally spaced visible nodes to obtain points of $\bhmm{4}$ at which we know the invariants of \ar{res-bhmi4} will vanish.

Given these extensions, one can hope that techniques similar to those used here could elucidate the algebraic statistics and geometry of HMMs with any number of hidden states as well.

\bibliographystyle{chicago}
\bibliography{critch-bhmm}{}
\end{document}